\documentclass[11pt, a4paper, reqno]{amsart}

\usepackage{amsmath,amsthm,amssymb}
\usepackage[margin=1.4in]{geometry}
\usepackage[abs]{overpic}

\usepackage{stackengine}
\usepackage{comment}
\usepackage{enumerate}
\usepackage{tikz}
\usetikzlibrary{cd, arrows}
\usepackage[arrow,matrix,curve,cmtip,ps]{xy}
\usepackage{pb-diagram}
\usepackage{pb-xy}
\usepackage{subcaption}
\usepackage{xcolor}

\definecolor{darkblue}{rgb}{0,0,0.6}

\usepackage[breaklinks, pdftex, ocgcolorlinks,colorlinks=true, citecolor=darkblue, filecolor=darkblue, linkcolor=darkblue, urlcolor=darkblue]{hyperref}
\usepackage[capitalize,noabbrev]{cleveref}

\makeatletter
\newtheorem*{rep@theorem}{\rep@title}
\newcommand{\newreptheorem}[2]{%
	\newenvironment{rep#1}[1]{%
		\def\rep@title{#2 \ref{##1}}%
		\begin{rep@theorem}}%
		{\end{rep@theorem}}}
\makeatother

\newtheorem{proposition}{Proposition}[section]
\newtheorem{theorem}[proposition]{Theorem}
\newtheorem{corollary}[proposition]{Corollary}
\newtheorem{lemma}[proposition]{Lemma}

\newreptheorem{theorem}{Theorem}
\newreptheorem{lemma}{Lemma}
\newreptheorem{proposition}{Proposition}
\newreptheorem{corollary}{Corollary}

\theoremstyle{definition}
\newtheorem{definition}[proposition]{Definition}
\newtheorem{question}[proposition]{Question}

\theoremstyle{remark}
\newtheorem{remark}[proposition]{Remark}

\newtheorem*{remark*}{Remark}

\numberwithin{equation}{section}

\newcommand{\sm}{\setminus}

\newcommand{\Q}{\mathbb{Q}}

\newcommand{\R}{\mathbb{R}}
\newcommand{\Z}{\mathbb{Z}}

\newcommand{\ul}{\underline}
\newcommand{\dvol}{\textrm{d}vol}
\newcommand{\vol}{\textrm{Vol}}

\newcommand{\F}{\mathbb{F}}
\newcommand{\frs}{\mathfrak{s}}
\newcommand{\frS}{\mathfrak{S}}

\newcommand{\cD}{\mathcal{D}}
\newcommand{\cH}{\mathcal{H}}

\newcommand{\cM}{\mathcal{M}}

\DeclareMathOperator{\Sym}{Sym}

\DeclareMathOperator{\spin}{spin}

\begin{document}

\title[Complete PSC 4-manifolds]{Complete Riemannian $4$-manifolds with uniformly positive scalar curvature}

 \author[O.\ Chodosh]{Otis Chodosh}
\address{Department of Mathematics, Bldg.\ 380, Stanford University, Stanford, CA 94305, USA}
 \email{ochodosh@stanford.edu}

 \author[D.\ Maximo]{Davi Maximo}
 \address{Department of Mathematics, University of Pennsylvania, Philadelphia, PA 19104, USA}
 \email{dmaxim@math.upenn.edu}

\author[A.\ Mukherjee]{Anubhav Mukherjee}
\address{Department of Mathematics, Princeton University, Princeton, NJ 08540, USA}
\email{anubhavmaths@princeton.edu}

\begin{abstract}
We obtain topological obstructions to the existence of a complete Riemannian metric with uniformly positive scalar curvature on certain (non-compact) $4$-manifolds. In particular, such a metric on the interior of a compact contractible $4$-manifold uniquely distinguishes the standard $4$-ball up to diffeomorphism among Mazur manifolds and up to homeomorphism in general. 

We additionally show there exist uncountably many exotic $\R^4$'s that do not admit such a metric and that any (non-compact) tame $4$-manifold has a smooth structure that does not admit such a metric. 
\end{abstract}

\maketitle

\section{Introduction}\label{sec:intro}

This article is concerned with the relationship between the \emph{scalar curvature} of a Riemannian manifold and the topology of the underlying smooth manifold. Scalar curvature (defined in \eqref{eq:defi-scalar-curv}) $R\in C^\infty(M)$ is the weakest curvature invariant of a Riemannian manifold $(M,g)$  but the existence of positive scalar curvature can place strong topological constraints on the underling topology. We survey some known results along these lines in Section \ref{sec:overview-psc} and Appendix \ref{app:survey}.   

One easily sees that for $n\geq 3$, the standard $\mathbb{R}^n$ admits a complete metric with uniformly positive  scalar curvature $R\geq 1$ by capping off a half-cylinder $(-\infty,0] \times S^{n-1}$ with a hemisphere $S^n_+$ (and then smoothing). A natural question is whether or not the existence of such a metric \emph{distinguishes} the standard $\mathbb{R}^n$ from other contractible manifolds. Indeed, for 3-manifolds this is true: work of Chang--Weinberger--Yu \cite[Theorem 1]{CWY10} shows that $\mathbb{R}^3$ is the only contractible $M^3$ that admits a complete Riemannian metric with uniformly positive scalar curvature (cf.\ Remark \ref{rema:CWY-result}). See also \cite{BBMM,Dong,wang2023topology}. 

Every non-compact smooth 4-manifold admits a (possibly incomplete) metric of uniformly positive scalar curvature (c.f. \cite[Theorem 0.1]{Rosen06}). Thus in this article, we study the existence of a \emph{complete} Riemannian metric with uniformly positive scalar curvature on non-compact 4-manifolds. We are able to completely resolve the classification problem in the case that $M$ is assumed to be the interior of a \emph{Mazur manifold} \cite{mazur}, namely $M$ is the interior of a compact contractible smooth $4$-manifold  with boundary $W$ admitting a (smooth) handle decomposition with one 0-handle, one 1-handle, and one 2-handle. 

\begin{theorem}\label{theo:Mazur}
Suppose that $M$ is the interior of a Mazur manifold $W$. If $M$ admits a complete Riemannian metric with scalar curvature $R\geq 1$ then $W$ must be diffeomorphic to the $4$-ball $B^4$.
\end{theorem}
(Of course it follows from this that $M$ is diffeomorphic to the standard $\R^4$.) The main ingredient in the proof of Theorem \ref{theo:Mazur} is a complete characterization of the homeomorphism type of contractible tame 4-manifolds admitting a complete metric of uniformly positive scalar curvature as follows:
\begin{theorem}\label{thm:b4}
Suppose that $M$ is the interior of a compact contractible smooth 4-manifold with boundary $W$.  If $M$ admits a complete Riemannian metric with uniformly  scalar curvature $R\geq 1$ then $W$ must be homeomorphic to the $4$-ball $B^4$.
\end{theorem}

We note that Chang--Weinberger--Yu \cite{CWY10} previously observed (using different methods) that certain contractible 4-manifolds cannot admit such a metric. See also \cite{BW,CWY,CWY2,weinberger2023gromovs}.

Our methods also show that uniformly positive scalar curvature can detect some exotic smooth structures. While it is not known whether exotic 4-balls can exist, there exists uncountably many non-diffeomorphic $\mathbb{R}^4$'s \cite{df, Gomp93, taubesp}. We show that many do not admit a complete metric with uniformly positive scalar curvature:
\begin{theorem}\label{uncountable R^4}
    There exist uncountably many exotic $\R^4$ which do not admit a complete Riemannian metric with uniformly positive scalar curvature $R\geq1$.
\end{theorem}
The above constructed exotic $\R^4$'s are small exotic $\R^4$, i.e. they can be smoothly embedded in standard $\R^4$ and construction of such small exotic $\R^4$ is motivated from the work of Gadgil \cite{gadgil2009open}. Using the same techniques as the proof of \cref{uncountable R^4} we can also obstruct the existence of certain uniformly positive curvature conditions namely bi-Ricci curvature on large exotic $\R^4$'s i.e. those exotic $\R^4$ that cannot be smoothly embedded in standard $\R^4$. 

It is known from the work of Kazaras \cite{Kazaras} that every 3-manifold that admits a positive scalar curvature metric bounds a compact 4-manifold whose interior supports a complete uniformly positive scalar curvature metric (see \cref{construction}). In fact, for general tame $4$-manifolds, we show that it's always possible to change the smooth structure to prevent the existence of such a metric: 
\begin{theorem}\label{thm:compact}
Suppose $W$ is a smooth 4-manifold with boundary. Then the interior of $W$ admits a smooth structure that does not admit a complete Riemannian metric with uniformly positive scalar curvature $R\geq1$.
\end{theorem}

Motivated by our results, it is natural to ask: 
\begin{question}
Is the standard $\R^4$ the unique contractible smooth $4$-manifold that admits a complete metric of uniformly positive scalar curvature? 
\end{question}
This question was previously raised in \cite{CWY10}. 
\begin{remark}
The methods here are insensitive to the scalar curvature of on a compact set (cf.\ Remark \ref{rema:cpt-set}). For example, Theorem \ref{thm:compact} shows that for some smooth structure on the once punctured $K3$ manifold does not admit a complete metric of uniformly positive scalar curvature. However, we have been unable to determine whether or not the standard smooth structure has such a metric. (The standard $K3$ manifold does not admit a metric of positive scalar curvature by spin-theoretic obstructions \cite{Lich63}.) As pointed out to us by Simone Cecchini, if $X^n$ is a closed smooth simply connected spin manifold of dimension $n\geq 6$ then the standard smooth structure on a punctured $X$ admits a complete metric of uniformly positive scalar curvature by \cite[Theorem 2.1(3)]{RW:double}.
\end{remark}

In a related direction, we can show that the end Floer homology as defined by Gadgil \cite{gadgil2009open} (see Section \ref{section:endFloer}) can obstruct the existence of complete metrics of uniformly positive scalar curvature. 

\begin{theorem}\label{thm:end-floer-obstruction}
Let $X$ be a non-compact smooth 4-manifold and $\mathfrak{s}$ be an asymptotic $spin^c$ structure, such that the end Floer homology $HE(X,\mathfrak{s})$ is non-trivial for at least one end, then $X$ does not admit a complete metric of scalar curvature $R\geq 1$.
\end{theorem}

We are also able to obstruct the existence of complete metrics with uniformly positive scalar curvature by utilizing contact and symplectic geometry. For more details, refer to the technical result \cref{non-trivial HE}.

\subsection{On \texorpdfstring{$R>0$}{R>0}} We note that the condition that $M$ admits a complete metric of positive scalar curvature $R>0$ (without the strict positive assumption) is much less restrictive. For example, the product of $S^1$ with a positively curved metric on $\R^2$ (e.g. a paraboloid) yields such a metric on the genus-one handlebody (which does not admit a complete metric of uniformly positive curvature by e.g.\ combining \cref{coro:boundary-tame} with \cref{prop:dominate2d} below). 

We note that even in $3$-dimensions, the classification of manifolds that admit a complete metric of positive curvature is a well-known open problem (cf.\ \cite[\# 27]{Yau:problem}). Some progress in the contractible case has been achieved (e.g.\ genus one contractible $3$-manifolds such as the Whitehead manifold does not admit such a metric \cite{Wang,Wang2}) but it's still unknown if $\R^3$ is the unique contractible $3$-manifold admitting such a metric. It's also unknown whether or not the genus-two (or higher) handlebody admits such a metric. We additionally refer to \cite{ChCh24,zhu2021,Chen:SYS,CRZ23,chen2024positive,lott2024obstructions,Zhu:calabi} for other results in this direction.

\subsection{Outline of techniques} 
The main idea of this article is to combine Gromov's $\mu$-bubble technique (cf.\ Section \ref{sec:mu-bubble}) with techniques from  contact and symplectic geometry as well as Floer theory. The $\mu$-bubble technique generalizes and localizes the Schoen--Yau stable minimal hypersurface method, and is particularly effective in the setting of non-compact manifolds. The Schoen--Yau minimal surface method was previously combined with some gauge theoretic methods in \cite{Lin,LRN,KM, KTJtop} to study whether or not certain (rational) homology $S^1\times S^3$ admit metrics of positive scalar curvature. 

The proofs of Theorems \ref{theo:Mazur} and \ref{thm:b4} can be roughly described as follows. The $\mu$-bubble technique in conjunction with the classification of positive scalar curvature $3$-manifolds places strong constraints (cf.\ \cref{coro:boundary-tame}) on the topology of $\partial W$ if $W$ is a smooth compact $4$-manifold whose interior admits a complete Riemannian metric of uniformly positive scalar curvature. On the other hand, if $W$ is assumed to be contractible then $\partial W$ must also be an integral homology sphere. These facts can be combined (cf.\ \cref{coro:poinc-PSC}) to prove that $\partial W$ is diffeomorphic to either $S^3$ or 
 $(\#_{\ell=1}^L P) \# (\#_{m=1}^M -P)$ for $P$ the Poincar\'e homology sphere (and $-P$ an oppositely oriented copy).  At this point, the assertions follow from deep results in topology. 
 
 In the proofs of Theorems \ref{uncountable R^4}, \ref{thm:compact}, and \ref{thm:end-floer-obstruction} we demonstrated how contact and symplectic geometry, along with Floer theory, come together to obstruct the existence of a special type of exhaustion. This, in turn, obstructs the existence of metrics with uniformly positive scalar curvature.

\subsection{Organization} 
In Section \ref{sec:overview-psc} we recall several classification results concerning positive scalar curvature that we will use later. Section \ref{sec:mu-bubble} contains a discussion of the $\mu$-bubble method. The end Floer homology obstruction is discussed in Section \ref{section:endFloer}. Then in Section \ref{sec:exotic} we construct various examples of exotic smooth structures that do not admit complete metrics with uniformly positive scalar curvature. Finally, in Appendix \ref{app:survey} we survey some general classification results of positive scalar curvature and in Appendix \ref{perturbedHF} we discuss perturbed Heegaard Floer homology. 

\subsection{Acknowledgements} We are grateful to Ian Agol, Simone Cecchini, Misha Gromov, Bob Gompf, Hokuto Konno, Chao Li, Ciprian Manolescu, Daniel Ruberman, Masaki Taniguchi, Clifford Taubes, Shmuel Weinberger, and Ian Zemke for helpful conversations related to this work. We are especially grateful to Bob Gompf for discussing his previous work and for detailed comments on a previous draft.  O.C. was supported by Terman Fellowship and NSF grant DMS-2304432. D.M. was partially supported by NSF grant DMS-191049. A.M. was supported in part by NSF grant DMS-2405270.

\section{Overview of PSC classification results}\label{sec:overview-psc}
In this section we review some results concerning the topology of (closed) manifolds admitting metrics with positive scalar curvature (PSC). A classification of smooth closed (compact, no boundary) oriented Riemannian manifolds $(M^n,g)$ with positive scalar curvature is available when $n=2,3$: 
\begin{itemize}
\item When $n=2$, $M$ is diffeomorphic to $S^2$. 
\item When $n=3$, $M$ is diffeomorphc to $(\#_{j=1}^J S^3/\Gamma_j) \# (\#_{k=1}^K S^2 \times S^1)$, where $\Gamma_j$ is a fixed-point free finite subgroup of $SO(4)$  \cite{perelman} (see also \cite{GrLa83,ScYa79b,ScYa79a}). 
\end{itemize}
In fact, we will need the following (well-known) mapping version of these results. It would be possible to state these results in the non-orientable case as well (with appropriate modifications) but we will not need this in the sequel. 

\begin{proposition}\label{prop:dominate2d}
Suppose that $X^2$ is a closed connected oriented surface and $(M^2,g)$ is a closed oriented surface with a Riemannian metric of positive scalar curvature. If there is a map of non-zero degree $M\to X$ then $X$ is diffeomorphic to $S^2$. 
\end{proposition}
\begin{proof}
Since $R=2K$ in $2$-dimensions, where $K$ is the Gaussian curvature, then the Gauss--Bonnet theorem implies that   $\chi(M) > 0$ and thus $M$ is diffeomorphic to $S^2$. The only $X$ with $S^2 \to X$ non-zero degree is $S^2$ (any other $X$ has $\pi_2(X) = 0$ by considering the universal cover, so the map would have to be null-homologous). 
\end{proof}

\begin{proposition}\label{prop:dominate3d}
Suppose that $X^3$ is a closed connected oriented $3$-manifold and $(M^3,g)$ is a closed oriented 3-manifold with a Riemannian metric of positive scalar curvature. If there is a map of non-zero degree $M\to X$ then $X$ is diffeomorphic to $S^3$ or else a connected sum of the form $(\#_{j=1}^J S^3/\Gamma_j) \# (\#_{k=1}^K S^2 \times S^1)$.
\end{proposition}
\begin{proof}
This follows from \cite{GrLa83,ScYa79a,perelman} as explained in \cite{AgolOverFlow}. We sketch an alternative approach. By Perelman's proof of the Poincar\'e conjecture and the Kneser--Milnor prime decomposition, if $X$ is not of the asserted form, then $X = X' \# K$ with $K$ a closed aspherical (having contractible universal cover) $3$-manifold. As such, by composing with the map crushing $X'$ to a point, it suffices to rule out the case that that $M$ admits a non-zero degree map $f: M \to K$ for $K$ a closed oriented aspherical $3$-manifold. 

This contradicts known results about positive scalar curvature and aspherical manifolds as described in \cite[Proposition 7.23]{Chodosh21} (cf.\ \cite{ChCh24}). We sketch the proof for completeness. One may lift to $\tilde f : \hat M \to \tilde K$ proper of non-zero degree, where $\tilde K$ is the universal cover and $\hat M$ is an appropriate cover (cf.\ \cite[Lemma 18]{CCL23} or \cite[Lemma 7.22]{Chodosh21}). We can then find linked embedding of $\R$ and $S^1$ into $\tilde K$ at a distance $L\gg 1$ apart (this uses $\tilde K$ contractible). On the other hand, since $(\hat M,\hat g)$ has uniformly positive scalar curvature (being a cover of a closed Riemannian manifold of positive scalar curvature) the preimage of the $S^1$ (perturbing $f$ slightly) can be filled by a minimal disk of bounded in-radius by diameter bounds for stable minimal hypersurfaces in positive scalar curvature \cite{SY:diameter}. After pushing this filling disk back to $\tilde K$, any point in the intersection with the linked $\R$ must have bounded distance to the $S^1$. This contradicts the fact that we can take the distance $L\gg1$ arbitrarily large. 
\end{proof}

The following corollary will be used in the sequel. Although it follows from known results, we could not find this exact statement in the literature and expect it to be useful in other settings. We begin by recalling the binary icosahedral group $I^* = \langle s,t | (st)^2 = s^3 = t^5 \rangle$. Letting $s,t$ act on $\R^4$ as quaternions $s=\frac 12 (1+i+j+k),t=\frac 12(\varphi+\varphi^{-1}i+j)$ (where $\varphi = \frac{\sqrt{5}-1}{2}$) one may check that $I^*\subset SO(4)$ acts on $S^3$ without fixed points. Thus $P=S^3/I^*$ is a spherical spaceform, so admits a metric of constant positive sectional (and thus scalar curvature). Writing $-P$ for $P$ with the opposite orientation, clearly $-P$ admits positive scalar curvature as well. Since positive scalar curvarture is preserved under connected sum (in dimensions $\geq 3$) we see that $(\#_{\ell=1}^L P) \# (\#_{m=1}^M -P)$ admits a metric of positive scalar curvature for any $L,M\geq 0$. On the other hand, it's well known that $H_*(P;\Z) = H_*(S^3;\Z)$, i.e.\ $P$ is an oriented integer homology sphere (this follows from the fact that $I^*$ is a perfect group, combined with Poincar\'e duality). 

Since the connected sum of integral homology spheres is again an integral homology sphere (this follows from Mayer--Vietoris) we conclude that $(\#_{\ell=1}^L P) \# (\#_{m=1}^M -P)$ is an integral homology sphere that admits a metric of positive scalar curvature.

\begin{corollary}\label{coro:poinc-PSC}
Suppose that $(M^3,g)$ is a closed oriented $3$-manifold with a Riemannian metric of positive scalar curvature. Suppose also that $M$ is an integral homology sphere, $H_*(M;\Z) = H_*(S^3;\Z)$. Then $M$ is diffeomorphic to either $S^3$ or the connected sum of Poincar\'e homology spheres $(\#_{\ell=1}^L P) \# (\#_{m=1}^M -P)$. 
\end{corollary}
\begin{proof}
By the classification of PSC $3$-manifolds (see Proposition \ref{prop:dominate3d}) we conclude that $M$ is diffeomorphic to $S^3$ or 
\[
(\#_{j=1}^J S^3/\Gamma_j) \# (\#_{k=1}^K S^2 \times S^1).
\]
Since $M$ is oriented, we can use Mayer--Vietoris to conclude that
\[
H_i((\#_{j=1}^J S^3/\Gamma_j) \# (\#_{k=1}^K S^2 \times S^1);\Z) = (\oplus_{j=1}^j H_i(S^3/\Gamma_j);\Z) \oplus (\oplus_{k=1}^K H_i(S^2 \times S^1;\Z)
\]
for $i=1,2$. Since $H_1(S^2\times S^1;\Z) = \Z$ we see that $K=0$. Finally,  from the classification of fixed-point free finite subgroups of $SO(4)$ one may check that the only one with trivial abelianization is $I^*$ (and the trivial group). See \cite[Theorem 2]{Kervaire}. This proves the assertion. 
\end{proof}

In dimensions $\geq 4$, a complete classification of closed manifolds admitting positive scalar curvature is not known. Several partial classification results have been obtained. We survey some of the techniques and results in \cref{app:survey}.

We briefly discuss the following result recently obtained by  R\"ade \cite[Proposition 2.17]{Rade23} (for $n\geq 6$) using surgery theory (see also \cite{GL80,Stolz,CRZ23}) since it will be referenced in the sequel:
\begin{proposition}\label{prop:Rade}
For $n\in \{3,4\} \cup\{6,7,\dots\}$, suppose that $Y$ is a $(n-1)$-dimensional closed connected oriented smooth manifold and $\Sigma \subset Y\times \R$ is a closed embedded separating hypersurface. If $\Sigma$ admits positive scalar curvature, then so does $Y$. 
\end{proposition}
The projection map $Y\times \R \to Y$ restricts to $\Sigma$ to yield a degree $1$ map $\Sigma\to Y$. Thus, Proposition \ref{prop:Rade} is a consequence of Propositions \ref{prop:dominate2d} and \ref{prop:dominate3d} when $n=3,4$. For the results in this article concerned with $4$-manifolds, we'll only need this version of Proposition \ref{prop:Rade} (as opposed to the high dimensional surgery result from \cite{Rade23}).   

\begin{remark} \label{rema:counterex-S1} We emphasize that there is a counterexample to Proposition \ref{prop:Rade} with $n=5$, see \cite[Remark 1.25]{Rosen06}. 
\end{remark}
\section{Exhaustions via \texorpdfstring{$\mu$}{mu}-bubbles}\label{sec:mu-bubble}

The following result is due to Gromov. After describing some consequences, we will explain the proof for completeness. 

\begin{proposition}[{\cite[\S 3.7.2]{Gromov23}}]\label{prop:gromov-mu-bubble-exhaustion}
For $3 \leq n \leq 7$ suppose that $(M^n,g)$ is a complete Riemannian manifold with scalar curvature $R\geq 1$. Then there's an exhaustion $\Omega_1\subset \Omega_2\subset \Omega_3\dots$ with $M = \cup_{i=1}^\infty\Omega_i$ where the $\Omega_i$'s are compact codimesnion-0 smooth submanifolds with smooth boundaries $\partial\Omega_i$ such that the $(n-1)$-manifolds $\partial\Omega_i$ admit metrics of positive scalar curvature. 
\end{proposition}

\begin{remark}\label{rema:cpt-set}
As will be clear from the proof, the same fact would hold with the weaker requirement that $R\geq 1$ outside of a compact set. 
\end{remark}

The dimensional restriction $n\leq 7$ is due to the potential presence of singularities in area-minimizing hypersurfaces in $8$-dimensional (and higher) manifolds. One may hope that with technical improvements in the study of generic regularity of (generalized) area minimizing hypersurfaces, one might be able to remove this condition (cf.\ \cite{HardtSimon,chodosh2023generic}). 

\begin{remark}\label{rema:CWY-result}
We note that Proposition \ref{prop:gromov-mu-bubble-exhaustion} (and the fact that $S^2$ is the only oriented closed $2$-manifold admitting positive scalar curvature) implies that if $(M^3,g)$ is a complete oriented Riemannian manifold with uniformly positive scalar curvature, then it admits an exhaustion by $\Omega_i$ with $\partial \Omega_i$ the disjoint union of $S^2$'s. For example, this implies that $\R^3$ is the only contractible $3$-manifold that admits such a metric, as originally proven in \cite{CWY10}. 
\end{remark}

We have the following important consequences of Proposition \ref{prop:gromov-mu-bubble-exhaustion}. 
\begin{corollary}\label{coro:boundary-tame}
For $n \in \{3,4,6,7\}$, suppose that  $W$ is a compact smooth $n$-manifold with boundary and that some component of $\partial W$ does not admit a Riemannian metric with positive scalar curvature. Then the interior of $W$ does not admit a complete Riemannian metric of with uniformly positive scalar curvature $R\geq 1$. 
\end{corollary}
\begin{remark}
The restrictions on the dimension is enforced by Propositions \ref{prop:Rade} and \ref{prop:gromov-mu-bubble-exhaustion}. (See Remark \ref{rema:counterex-S1}.) 
\end{remark}

 \begin{remark} \label{construction}
     We note that for $n\in\{3,4\}$ the converse to Corollary \ref{coro:boundary-tame} holds in the following sense. Suppose that $Y^{n-1}$ is a compact (smooth) oriented $(n-1)$-manifold that admits a metric of positive scalar curvature. Then there exists a compact smooth $n$-manifold $W$ with $\partial W$ diffeomorphic to $Y$ so that the interior of $W$ admits a Riemannian metric with uniformly positive scalar curvature. When $n=3$, since $Y^2$ oriented compact admitting positive scalar curvature implies $Y^2$ is diffeomorphic to $S^2$ this easily follows from the existence of a metric with uniformly positive scalar curvature on the $3$-ball. When $n=4$, this follows from work of Kazaras \cite[Theorem B]{Kazaras} after noting that the cobordisms constructed there are cylindrical near the boundary and thus smoothly glue to a half-cylinder. 
 \end{remark}

\begin{proof}[Proof of Corollary \ref{coro:boundary-tame}]
Let $M$ denote the interior of $W$. The collar neighborhood theorem yields a neighborhood of infinity $U$ diffeomorphic to $\partial W \times \R$. For $i$ sufficiently large, the exhaustion obtained in Proposition \ref{prop:gromov-mu-bubble-exhaustion} will have $\partial \Omega_i$ a separating hypersurface in $U$. The assertion follows by combining Propositions \ref{prop:Rade} and \ref{prop:gromov-mu-bubble-exhaustion}. 
\end{proof}

Combining all these results, we derive the following corollary, which will be crucial for proving our main theorems.

\begin{corollary} \label{exhaustion}
Suppose $(M^4,g)$ is complete smooth oriented  4-manifold with scalar curvature  $R \geq 1$, then there's an exhaustion $\Omega_1 \subset \Omega_2\subset\dots$ so that each component of $\partial\Omega_i$ is a PSC $3$-manifold, namely of the form $(\#_{j=1}^J S^3/\Gamma_j) \# (\#_{k=1}^K S^2 \times S^1)$.
\end{corollary}
\begin{proof}
    Combine  Propositions \ref{prop:gromov-mu-bubble-exhaustion} and \ref{prop:dominate3d}. 
\end{proof}

\subsection{\texorpdfstring{$\mu$}{mu}-bubbles and the proof of Proposition \ref{prop:gromov-mu-bubble-exhaustion}}   The foundational method goes back to the influential work of Schoen--Yau who studied the second variation of area at area minimizing hypersurfaces \cite{ScYa79a,ScYa79b}. They observed that if $\Sigma^{n-1}$, $n\geq3$, is a 2-sided stable minimal hypersurface inside of a Riemannian manifold $(M^n,g)$ with positive scalar curvature then a suitable conformal deformation of the induced metric $g|_\Sigma$ yields a metric of positive scalar curvature on $\Sigma$. This gives a topological obstruction to positive scalar curvature when paired with an appropriate existence result for minimal hypersurfaces. 

A careful examination of Schoen-Yau argument reveals that it relies mostly on the second variation of area and not on minimality per se. Following \cite{ChCh24}, we describe next an idea of Gromov \cite{Gromov23} in which one gives up minimality by considering instead modified area functional. The key point is that the minimization of these functions can be \emph{localized} and the second variation of the functional can still obstruct the positivity of scalar curvature in certain cases. 

For $3\leq n \leq7$, suppose $(M^n,g)$ is a Riemannian manifold with non-empty boundary $\partial M$. Assume $\partial M=\partial_{-} M \sqcup \partial_{+} M$ is labeling of the boundary components such that each of them is non-empty and let $h$ be a smooth function on $\mathring M$ such that 
$$h\rightarrow+\infty~\textrm{on}~\partial_{-}M, \quad  h\rightarrow-\infty~\textrm{on}~\partial_{+}M.$$
Let $\Omega_0$ be a Caccioppoli set with smooth boundary $\partial \Omega_0\subset \mathring{M}$. For all Caccioppoli sets $\Omega$ in $M$ with $\Omega \triangle \Omega_0 \Subset \mathring{M} $, we consider the $\mu$-bubble functional: 
\begin{eqnarray}
\mathcal{A}(\Omega) = |\partial^\ast\Omega| - \int_M (\chi_\Omega-\chi_{\Omega_0}) h~\dvol.
\end{eqnarray}

A minimizer of $\mathcal A$ is called a \textit{$\mu$-bubble}. As the next proposition shows, they are easier to construct than stable minimal hypersurfaces (see Proposition 12 of \cite{ChCh24} for a proof).

\begin{proposition}
There exists a smooth minimizer $\Omega$ for $\mathcal{A}$ such that $\Omega\triangle \Omega_0$ is compactly contained in the interior of $M$.
\end{proposition}

The first and second variations of the $\mu$-bubble functional are as follows. If $\Omega_t$ is a smooth 1-parameter family of regions with $\Omega_0=\Omega$ and normal speed $\varphi$ at $t=0$, then 
$$\frac{d}{dt} \mathcal{A}(\Omega_t) =\int_{\partial \Omega_t} (H-h) \varphi $$
where $H$ is the mean curvature of $\partial \Omega_t$. In particular, a $\mu$-bubble must satisfy $H=h$ along $\partial \Omega=\Sigma$. 

\begin{remark}The first variation can be used to formally explain why a minimizer should exist: $\partial_-M$ has mean curvature $H_{\partial_-M} - h = - \infty$ and thus acts as a strict barrier for a minimizing sequence for $\mathcal{A}(\cdot)$. Similar considerations hold for $\partial_+M$, and thus a minimizing sequence $\Omega_j$ has $\Omega_j\Delta\Omega_0$ compactly contained in the interior of $M$. Thus, standard arguments in geometric measure theory allow one to pass to a weak limit to find $\Omega$ minimizing $\mathcal{A}(\cdot)$ so that $\Omega\Delta\Omega_0$ compactly contained in the interior of $M$. Since $\mathcal{A}(\cdot)$ is a perturbation of the area functional (at small scales in the interior of $M$), standard regularity theory implies that $\partial\Omega=\Sigma$ is a smooth compact submanifold contained in the interior of $M$. (When $n\geq 8$ there could exist singularities along $\partial\Omega$, so we cannot directly apply this method without further modification.)
\end{remark}

Assuming that $\partial\Omega$ satisfies $H=h$ along $\Sigma$, the second variation is
\begin{eqnarray}\label{eq:stability-mu-bubb}
\dfrac{d^2}{dt^2}\Big|_{t=0} \mathcal{A}(\Omega_t) = \int_{\Sigma} |\nabla_\Sigma \varphi|^2- \frac{1}{2}(R_M  - R_\Sigma + |A|^2 + h^2+2\langle \nabla_M h,\nu \rangle)\varphi^2  
\end{eqnarray}
See e.g.\ Lemma 14 \cite{ChCh24} (one must take $u=1$ and rearrange the terms slightly). In the above formula $\mathring A$ is the trace-free part of the second fundamental form of $\Sigma$ (this will not matter since we'll discard it using $|\mathring A|^2 \geq 0$ in the sequel), $R_M$ is the scalar curvature of $M$ and $R_\Sigma$ is the scalar curvature of the induced metric on $\Sigma$. 

The minimizer $\Omega$ obtained above will satisfy the stability inequality: $\tfrac{d^2}{dt^2}|_{t=0} \mathcal{A}(\Omega_t)\geq 0$ for any $\varphi\in C^\infty(\Sigma)$. 
\begin{proposition}\label{prop:bandestimate}
For $3\leq n\leq 7$, suppose $(M^n,g )$ is a Riemannian manifold with non-empty boundary and scalar curvature $R\geq\Lambda>0$. Assume $\partial M=\partial_{-} M \sqcup \partial_{+} M$ is a labeling of the boundary components such that each of them is non-empty. There exists a constant $D=D(\Lambda)>0$ such that if the distance $d(\partial_{-} M, \partial_{+} M)>D$, then in the interior of $M$ there must be a there exists a smoothly embedded closed 2-sided hypersurface  $\Sigma^{n-1}$ which itself admits a metric with positive scalar curvature. 
\end{proposition}
The estimate for $D(\Lambda)$ given below could be strengthened to give a sharp/explicit estimate, but we will not bother with this here since it will not matter in the sequel. See e.g.\ \cite{Gromov:pos-mac,Gromov23,Chodosh21} and references therein. 
\begin{proof}
It suffices to assume there is $\rho : M\to[-
\frac{\pi}{2},\frac{\pi}{2}]$ smooth with $\textrm{Lip}(\rho) \leq \frac \pi D$, $\rho^{-1}(\pm\frac \pi 2)= \partial_\pm M$. Indeed, one may smooth out (and rescale) the distance function from $\partial_-M$ to obtain such $\rho$ with $\partial_-M = \rho^{-1}(-\frac\pi 2)$ and $\frac\pi 2$ a regular value. 

Then we can define $M$ to be $\rho^{-1}([-\frac\pi2,\frac\pi2])$. Then define $h(x) = -\frac \pi D \tan(\rho(x))$ on the interior of $M$. Note that
\[
h^2 + 2\langle\nabla_M h,\nu\rangle \geq h^2 - 2 |\nabla_M h| \geq \frac{\pi^2}{D^2}\left( \tan(\rho(x))^2 - \sec(\rho(x))^2\right) = - \frac{\pi^2}{D^2}. 
\]
Using this in the stability inequality \eqref{eq:stability-mu-bubb} (and using $R_M\geq \Lambda$, and discarding the $|\mathring A|^2$ term) we obtain 
\[
\int_\Sigma |\nabla_\Sigma \varphi|^2 +\frac 12 R_\Sigma \varphi^2 - \frac 12\left(\Lambda - \frac{\pi^2}{D^2}\right)\varphi^2  \geq 0.
\]
As long as $D\geq D(\Lambda)>0$ is sufficiently large, the term $\frac 12\left(\Lambda -  \frac{\pi^2}{D^2}\right)$ will be uniformly positive. To be definite we can take $D(\Lambda)^2 = \frac{2\pi}{\Lambda}$ so this term will be $\geq \frac{\Lambda}{4}$. 

Thus
\begin{equation}\label{eq:spectral-R-mubub}
\int_\Sigma |\nabla_\Sigma \varphi|^2 +\frac 12 R_\Sigma \varphi^2 \geq \frac{\Lambda}{4} \int_\Sigma \varphi^2. 
\end{equation}
When $n=3$, $\dim\Sigma =2$, so $R_\Sigma = 2K_\Sigma$ is the Gaussian curvature and we can take $\varphi=1$ on any connected component $\Sigma'\subset \Sigma$ to find
\[
2\pi\chi(\Sigma') = \int_{\Sigma'} K_\Sigma > 0,
\]
so $\Sigma'$ is a sphere (or  possibly projective plane if $M$ is nonorientable) and thus admits positive scalar curvature. 

When $n\geq 4$ we observe that
\[
\frac{2(m-2)}{m-3} \geq 1
\]
so using $|\nabla_\Sigma \varphi|^2\geq 0$ we find that 
\[
\int_\Sigma \frac{2(m-2)}{m-3}|\nabla_\Sigma \varphi|^2 +\frac 12 R_\Sigma \varphi^2 \geq \frac{\Lambda}{4} \int_\Sigma \varphi^2,
\]
i.e. the operator
\[
L_\Sigma := -\frac{4(m-2)}{m-3}\Delta_\Sigma +  R_\Sigma
\]
is a positive operator (e.g.\ all eigenvalues are positive). Letting $u>0$ denote the first eigenfunction of $L_\Sigma$ (so $L_\Sigma u = \lambda u > 0$), a computation (see p.\ 9 in \cite{ScYa79a}) implies that $\tilde g_\Sigma = u^{\frac{4}{n-3}}g_\Sigma$ has scalar curvature
\[
\tilde R = u^{\frac{m+1}{m-3}}Lu > 0. 
\]
This completes the proof. 
\end{proof}

We can now prove the exhaustion result.  
\begin{proof}[Proof of Proposition \ref{prop:gromov-mu-bubble-exhaustion}]
     Fix a compact region $M_1'$ with smooth boundary so that for some $p\in M_1'$, $d(p,\partial M_1') > D + \varepsilon$ for $D=D(\Lambda)>0$ as in Proposition \ref{prop:bandestimate}, and $\varepsilon>0$ small enough so that $B_\varepsilon(p)$ has smooth boundary. Let $M_1 = \overline{M_1'\setminus B_\varepsilon(p)}$ and set $\partial_-M_1 = \partial B_\varepsilon(p)$ and $\partial_+ M_1 = \partial M_1'$. Proposition \ref{prop:bandestimate} produces $\Omega\subset M_0$ so that $\partial\Omega = \partial B_\varepsilon(p) \cup \partial'\Omega$ and each component of $\partial'\Omega$ admits positive scalar curvature. Set $\Omega_1 = \Omega\cup B_\varepsilon(p)$. Assume that $\Omega_1\subset \dots\Omega_{i-1}$ have been chosen as in the statement of the theorem. Choose a smooth compact region $\Omega_{i-1}\subset M_i'$ with $d(\partial M_i',\Omega_{i-1}) > D$. As before, applying Proposition \ref{prop:bandestimate} to $M_i := \overline{M_i'\setminus \Omega_{i-1}}$ yields $\Omega \subset M_i$ whose boundary components in the interior of $M_i$ admit positive scalar curvature. Thus, setting $\Omega_i : =\Omega\cup \Omega_{i-1}$ completes the inductive step. 
\end{proof}

\section{Classifying uniformly PSC tame contractible \texorpdfstring{$4$}{four}-manifolds}
 We recall the following well-known fact which follows from Lefschetz duality: 
\begin{lemma}\label{lemm:bdry-contract}
If $W$ is a compact contractible topological $4$-manifold then $\partial W$ is an integral homology sphere. 
\end{lemma} 

The next result collects several deep results in topology together to obstruct connect sums of the Poincar\'e homology sphere $P$ from bounding a \emph{smooth} contractible $4$-manifold. (Note that by Freedman's work \cite{Freedman}, any integral homology sphere bounds a compact contractible topological $4$-manifold.)
\begin{proposition}\label{prop:poincare-bdry}
For $L,M\geq 0$ suppose that $W$ is a compact contractible smooth $4$-manifold with $\partial W$ diffeomorphic to $(\#_{\ell=1}^L P) \# (\#_{m=1}^M -P)$. Then $L=M=0$ $\partial W$ is diffeomorphic to $S^3$, and $W$ is homeomorphic to the $4$-ball. 
\end{proposition}
\begin{proof}
Since $\partial W$ is the boundary of a contractible smooth $4$-manifold, it follows that the Heegaard Floer $d$-invariant\footnote{Note that a homology $3$-sphere has a unique $\spin^c$ structure.} satisfies $d(\partial W) = 0$ by \cite[Theorem 1.2]{osd}. On the other hand, by \cite[Theorem 4.3 and Proposition 6.3]{osd} we find that 
\[
d((\#_{\ell=1}^L P) \# (\#_{m=1}^M -P)) = (L-M) d(P). 
\]
and $d(P) \neq 0$ by \cite[\S 8.1]{osd}. Thus $L=M$ so we see that either $\partial W$ is $S^3$ or $\#_{\ell=1}^L(P \# -P)$. The latter case cannot occur by the periodic end theorem of Taubes \cite{taubesp} which implies that if $Y$ is a homology $3$-sphere that bounds a negative definite $4$-manifold with non-diagonal intersection form (e.g. $Y= \#_{\ell=1}^L P$ bounds the boundary connect sum of the $E_8$-plumbing of spheres) then $Y \# -Y$ cannot bound a contractible $4$-manifold. 

Thus, $\partial W$ is diffeomorphic to $S^3$. The work of  Freedman \cite{Freedman} implies that $W$ is homeomorphic to the $4$-ball. This finishes the proof. 
\end{proof}

We now prove that a complete metric with uniformly positive scalar curvature distinguishes the ball up to homeomorphism among tame contractible $4$-manifold.

\begin{proof}[Proof of Theorem \ref{thm:b4}]
Consider $W$ a compact contractible smooth $4$-manifold and denote by $M$ the interior. Assume that $M$ admits a complete Riemannian metric of with uniformly positive scalar curvature $R\geq 1$. By Corollary \ref{coro:boundary-tame}, $\partial W$ admits positive scalar curvature. Furthermore, by Lemma \ref{lemm:bdry-contract}, $\partial W$ is an integral homology sphere. Thus by Corollary \ref{coro:poinc-PSC}, we see that $\partial W$ is either $S^3$ or   $(\#_{\ell=1}^L P) \# (\#_{m=1}^M -P)$. The assertion thus follows from Proposition \ref{prop:poincare-bdry}. 
\end{proof}

Using this, we can now consider the case of $W$ a Mazur manifold.
\begin{proof}[Proof of Theorem \ref{theo:Mazur}]
Assume that $W$ is a compact contractible smooth $4$-manifold admitting a smooth handle decomposition with one $0$-handle, one $1$-handle and one $2$-handle. Assume that the interior $M$ admits a complete Riemannian metric with uniformly positive scalar curvature $R\geq 1$. By Theorem \ref{thm:b4} (proven above) we see that $W$ is homeomorphic to the $4$-ball. 

We now observe that the $2$-handle is attached along a knot on the boundary of the $1$-handle which is $S^1\times S^2$. Consequently, we have obtained $S^3$ by performing surgery along a knot $K$ in $S^1\times S^2$. Hence, Gabai's property R theorem \cite{gabai} implies that $K$ is smoothly isotopic to the $S^1$ factor of $S^1\times S^2$. In the 4-dimensional handle picture, the attaching sphere of the $2$-handle intersects the belt sphere of the $1$-handle geometrically once, allowing us to cancel the $1$- and $2$-handle smoothly. Thus, $W$ must be diffeomorphic to a $4$-ball. This completes the proof. 
\end{proof}

\section{A general obstruction via end Floer homology} \label{section:endFloer}
 In order to obstruct the existence of complete positive scalar curvature metric in open 4-manifolds we will first recall the construction of end Floer homology. 

\begin{definition} Let $X$ be a smooth open 4-manifold. An asymptotic $\spin^c$ structure $\mathfrak{s}$ on $X$ is a $\spin^c$ structure defined on $X\sm K$ for some compact subset $K\subset X$. Two such $\spin^c$ structures $\mathfrak{s_1}$ and $\mathfrak{s_2}$ defined on $X\sm K_1$ and $X\sm K_2$ are said to be equal at infinity if there exists a compact set $K'$ which contains both $K_1$ and $K_2$ and $\mathfrak{s_1} |_{X\sm K'} = \mathfrak{s_2} |_{X\sm K'}.$
    
\end{definition}

 \begin{definition}Let $X$ be a smooth 4-manifold. We call $X_1\subset X_2\subset \cdots$ be an {\em{exhaustion}} if it satisfies the following properties: 
 \begin{itemize}
     \item[(i)] $\cup_i X_i = X$,
     \item[(ii)] $X_i$ is a smooth compact 4-manifold with boundary $Y_i$ for all $i$.
 \end{itemize} 
 Moreover, an exhaustion is called {\em{admissible}} if the map induced by the inclusion $H^1(X_{i+1}\sm X_i) \to H^1(Y_{i+1})$ is surjective for all $i$.
\end{definition}

\begin{remark}
    If $W = W_1\cup_Y W_2$ be a cobordism such that the induced map by inclusion $H^1(W_1)\to H^1(Y)$ is surjective then given $\mathfrak{s_i} \in spin^c(W_i)$ for $i=1,2$, there exists an unique $spin^c$ structure $\mathfrak{s}$ on $W$ which restricts to $\mathfrak{s_i}$ on $W_i$.
\end{remark}

\begin{remark} \label{admissible cobordism}
    Let $W$ be a cobordism from $Y_1$ to $Y_2$ which is obtained from $Y_1\times [0,1]$ in one of the following ways:
    \begin{itemize}
    
        \item Attach a $2-$handle along a knot $K\in Y_1 \times \{1\}$ which represents a primitive, non-torsion element in $H_1(Y_1)$ or,
        \item  $W$ is a rational-homology cobordism, i.e. $H_k(W,Y_i; \Q)= 0$ for all $k$ and $i=0,1$ or,
        \item Attach a $1-$handle along $Y_1\times \{1\}$,
    \end{itemize}
    then the induced map by the inclusion $H^1(W)\to H^1(Y_2)$ is surjective. 
\end{remark}

Consider a smooth open 4-manifold $X$ with an asymptotic $\spin^c$ structure $\mathfrak{s}$. Given an admissible exhaustion $X_1\subset X_2\subset \cdots$ of $X$, we can define the end Floer homology $HE(X, \mathfrak{s})$ as the direct limit of the reduced Floer homology groups $HF^+_\textnormal{red}(Y_i, \mathfrak{s}|_{Y_i})$, where the morphism is induced by the cobordisms $X_{i+1}\setminus \mathring{X_i}$. To be more precise, suppose $\omega$ is a $2-$form on $X\setminus K$ for some compact set $K\subset X$. 

  Now, consider the admissible cobordism $W_{ij}= X_j\setminus \mathring{X_i}$ from $Y_i$ to $Y_j$. Then we have the $\omega$-twisted induced map $\ul{F}^+_{W_{ij};\omega} : \ul{HF}_\textrm{red}^+(Y_i, \omega|_{Y_i}) \to \ul{HF}_\textrm{red}^+(Y_j, \omega|_{Y_j})$ (see \cref{perturbedHF}). The \textit{end Floer homology} $\ul{HE}(X,\mathfrak{s},\omega)$ is defined as the direct limit of $\ul{F}^+_{W_{ij};\omega}$. Note that if $Y$ is a rational homology sphere, then the Heegaard Floer homology with twisted coefficients is equivalent to the untwisted version. For further details on twisted Heegaard Floer theory, refer to \cite{twistedos-sz}.

 \begin{theorem}[Gadgil \cite{gadgil2009open}] \label{weldefined} Let $X$ be a smooth open 4-manifold and $\mathfrak{s}$ be an admissible $spin^c$ structure on $X$. Then, $\underline{HE}(X,\mathfrak{s};\omega)$ does not depend on the choices of admissible exhaustion. 
 \end{theorem}

We call a 3-manifold $Y$ an $L-$space if $\ul{HF}_\textrm{red}^+(Y,M)=0$ holds for every $\Z[H^1(Y,\Z)]$ module $M$. Notably, any 3-manifold that permits a positive scalar curvature metric also is an $L-$space \cite[Proposition 2.3]{l-space}. With this in mind, we have the following vanishing result of end Floer homology. 

\begin{remark}\label{hfhat}
Alternatively, we can also define a minus version of end Floer homology by considering all the groups and cobordism maps on $\ul{HF}^-_\textrm{red}$ and Gadgil's proof of \cref{weldefined} also valid for this and gives an invariant.
    
\end{remark}

\begin{theorem}\label{thm: vanishing HE}
    Let $X$ be a smooth open 4-manifold which admits an exhaustion $\Omega_1 \subset \Omega_2\subset\dots$ so that each component of $\partial\Omega_i$ is an $L-$space. Then $\ul{HE}(X,\mathfrak{s};\omega)= 0$ for all choice of asymptotic $\spin^c$ structures $\mathfrak{s}$ and 2-forms $\omega$.
\end{theorem}
\begin{proof}

Suppose $W$ is a cobordism between $Y_1$ and $Y_2$, with $Y'$ smoothly embedded in $W$, serving as an $L-$space that separates the two boundaries $Y_1$ and $Y_2$. Then, observe that the induced twisted cobordism map $\ul{F}^+_{W;\omega} : \ul{HF}_\textrm{red}^+(Y_1, \omega|_{Y_1}) \to \ul{HF}_\textrm{red}^+(Y_2, \omega|_{Y_2})$ factors through $\ul{HF}^+_\textrm{red}(Y',\omega|_{Y'}) =0$. Consequently, the image of $\ul{F}^+_{W;\omega}$ is $0$.

Now consider an admissible exhaustion $X_0\subset X_1\subset X_2\subset \cdots$ for $X$. As $X$ admits an exhaustion by $\Omega_1 \subset \Omega_2\subset\dots$, where each component of $\partial\Omega_i$ is an $L-$space, we can refine the admissible exhaustion to $X_{i_1}\subset X_{i_2}\subset \cdots$ of ${X_i}$'s such that in this refined exhaustion, for each cobordism $X_{i_j}\setminus \mathring{X}_{i_{j-1}}$, $\partial \Omega_l$ embeds smoothly in a boundary-separating manner for some $l$. Hence, the induced cobordism map is trivial, yielding $\ul{HE}(X,\mathfrak{s};\omega)= 0$.
\end{proof}


Now we are ready to prove \cref{thm:compact} which says that a non-compact 4-manifold with non-trivial end Floer homology cannot admit a complete positive scalar metric $R\geq 1$.

\begin{proof}[Proof of \cref{thm:end-floer-obstruction}]

If $X$ admits a complete positive scalar curvature metric $R\geq 1$, then by \cref{prop:gromov-mu-bubble-exhaustion}, it admits an exhaustion $\{\Omega_i\}$ such that $\partial \Omega_i$ admits a positive scalar curvature metric, and hence an L-space. Thus by \cref{thm: vanishing HE}, the end Floer homology satisfies $\ul{HE}(X,\mathfrak{s};\omega)= 0$ for all choice of asymptotic $\spin^c$ structures $\mathfrak{s}$ and 2-forms $\omega$, which is a contradiction. 
\end{proof}
\begin{remark} \label{HF-}
Since there is a natural and canonical isomorphism from $HF^+_\textnormal{red}\to HF^-_\textnormal{red}$, the above \cref{thm:end-floer-obstruction} is valid if we define end Floer homology as a limit of $HF^-$-flavor. 
    
\end{remark}

\section{Exotic smooth structures with no uniformly PSC metric}\label{sec:exotic}
In this section, we will use contact and symplectic geometry to produce many examples of smooth 4-manifolds that do not admit complete metric with uniformly positive scalar curvature.

\subsection{The idea of constructing exotic \texorpdfstring{$\R^4$}{R4}} Begin with a disjoint collection of smooth disks $D$ in $B^4$ whose boundary $\partial D \subset \partial B^4 = S^3$ gives a link. By removing a tubular neighborhood of $D$ from $B^4$ we obtain a 4-manifold  $B'$ whose boundary is the 3-manifold which is obtained from $S^3$ by doing $0$-surgery on the link $\partial D$. If we were to attach 2 handles on $B'$ along the meridians of $D$, we will recover $B^4$. However, if one were to glue instead (Stein) Casson handles \cite{casson, gompf} to these meridians of $D$. The resultant interior will be homeomorphic to $\mathbb{R}^4$ (since Casson handles are homeomorphic to open 2-handles \cite{Freedman}), but it may not be diffeomorphic to $\mathbb{R}^4$. For our purposes, we need to choose certain configurations of handles that can obstruct the existence of a metric with uniformly positive scalar curvature. In \cref{figure:ribbon} we describe a ribbon disk complement $B'$ which is obtained from deleting the standard disk $D$ bounded by Pretzel knot $P(-3,-3,3)$ and the dashed circle in \cref{figure:ribbon} is the meridian of $D$. We observe that with this specific choice, Gompf's Casson handle \cite{gompf} will yield exotic $\R^4$'s for which the existence of a metric with uniformly positive scalar curvature can be obstructed using contact geometry and Heegard Floer homology.

\subsection{Contact and Symplectic Topology} Let us review some fundamental concepts of contact and symplectic geometry that we will need for our constructions later on.

Let $\xi$ denote a contact structure on an oriented 3-manifold (sometimes $\xi$ can also be considered as kernel of a 1-form $\theta$ on $Y$). A knot $K$ contained in $(Y,\xi)$ is termed \textit{Legendrian} if the tangent space $T_p K$ lies within $\xi_p$ for every $p\in K$. In a contact manifold $(Y,\xi)$, a Legendrian knot $K$ possesses a standard neighborhood $N$ and a framing $fr_\xi$ determined by the contact planes. When $K$ is a null-homologous knot, it bounds a Seifert surface in the 3-manifold and the positive normal direction to the surface induces a framing, namely the Seifert framing, and the framing $fr_\xi$ relative to the Seifert framing represents the \textit{Thurston--Bennequin} number of $K$, denoted by $tb(K)$. We call a contact 3-manifold $(Y,\xi)$ \textit{over-twisted} if there exists a Legendrian unknot $K$ with $tb(K)=0$, i.e. $K$ bounds an embedded disk. And $(Y,\xi)$ is called \textit{tight} if it is not an over-twisted contact manifold. Performing a $(fr_\xi-1)$-surgery on $K$, i.e.\ by removing $N$ and attaching a solid torus to realize the desired surgery, results in a unique extension of $\xi|_{Y-N}$ over the surgery torus, maintaining tightness on the surgery torus. The resulting contact manifold is termed obtained from $(Y,\xi)$ by \textit{Legendrian surgery} on $K$. Furthermore, for a knot $K$ in $(S^3, \xi_{\text{std}})$, the {\it maximal Thurston--Bennequin number} is defined as the maximal value among all Thurston--Bennequin numbers for all Legendrian representations of $K$.

\begin{definition}
    
A \textit{symplectic cobordism} from the contact manifold $(Y_-,\xi_-)$ to $(Y_+,\xi_+)$ constitutes a compact symplectic manifold $(W,\omega)$ such that there exists a vector field $v$ near $\partial W$ pointing transversally inwards at $Y_-$ and transversally outwards at $Y_+$, so that the Lie derivative of $\omega$ satisfies $\mathcal{L}_v \omega= \omega$ and $\iota_v \omega |_{Y_\pm}$ represents a contact form of $\xi_\pm$. The boundary $Y_-$ is called \textit{concave} boundary and $Y_+$ is \textit{convex}  If $Y_-$ is empty, $(W,\omega)$ is referred to as a {\it symplectic filling} and if $Y_+$ is empty, then $(W,\omega)$ is called a {\it symplectic cap}.
\end{definition}

Our approach predominantly follows a technique for constructing symplectic cobordisms known as {\it Stein handle attachment} \cite{eliashberg, weinstein91}. This involves attaching 1- or 2-handles to the convex end of a symplectic cobordism to obtain a new symplectic cobordism with the modified convex end as follows: for a 1-handle attachment, the convex boundary undergoes a connected sum, potentially internal. On the other hand, a 2-handle is attached along a Legendrian knot $L$ with framing one less than the contact framing, leading to a Legendrian surgery on the convex boundary.

\begin{theorem}[\cite{eliashberg,weinstein91,eh02}]\label{cob}
Let $(W,\omega)$ be a symplectization of a contact 3-manifold $(Y,\xi=\ker \theta)$ , i.e., $(W= [0,1]\times Y, \omega= d(e^t\theta))$. Let $L$ be a Legendrian knot in $(Y,\xi)$, with $Y$ regarded as $Y\times { 1 }$. If $W'$ is obtained from $W$ by either attaching a 1-handle or attaching a 2-handle along $L$ with framing one less than the contact framing, then the upper boundary $(Y', \xi ')$ remains a convex boundary. Furthermore, if the 2-handle is attached to a symplectic filling of $(Y,\xi)$, then the resulting manifold would be a symplectic filling of $(Y',\xi')$.
\end{theorem}

\begin{remark}
    As in this above Theorem, attaching a 1-handle or 2-handle along a Legendrian knot $L$ with framing one less than the contact framing is called \textit{Stein} or $\textit{Weinstein}$ handle attachment.
\end{remark}

\begin{definition}
    An \emph{admissable Stein surface $(W,\omega)$ with a concave boundary $(Y,\xi)$} is a symplectic $4$-manifold with $(Y,\xi)$ concave in the sense described above, so that $(W,\omega)$ admits an admissible exhaustion $\{W_i\}$ where $W_0 = Y\times [0,1]$, and so that $W_j$ is obtained by attaching a Stein handle of index 1 or 2 on the convex boundary $(\partial W_{j-1}, \xi_{j-1})$.
\end{definition}

 Given a contact 3-manifold $(Y,\xi)$, Ozsváth and Szabó \cite{twistedos-sz} defined a corresponding Heegaard Floer homology invariant, namely the \textit{contact invariant} $\widehat{c}(\xi;\omega)$ as an element of the group $\ul{\widehat{HF}}(-Y,\omega)$. And let us denote $c^+(xi;\omega)$ as the image of $\widehat{c}(\xi; \omega)$ under the natural map $\ul{\widehat{HF}}(-Y,\omega)\to \ul{HF}^+(-Y,\omega)$ and $c^+_\textnormal{red}(\xi;\omega)$ as the image of $c^+(\xi;\omega)$ under the projection map $\ul{HF}^+(-Y,\omega)\to \ul{HF}^+_\textnormal{red}(-Y,\omega)$. 

\begin{theorem} \label{non-trivial HE}
    Let $(Y,\xi)$ be a contact 3-manifold such that reduced contact invariant $c^+_\textnormal{red}(\xi;\omega)\neq 0$ in $\ul{HF}^+_\textnormal{red}(-Y,\omega)$. If $(W,\omega)$ is an admissible Stein surface with a concave boundary $(Y,\xi)$, then for any smooth and compact 4-manifold $X$ with boundary $Y$, the interior of $X\cup_Y W$ doesn't admit a complete positive scalar curvature metric $R\geq 1$.
\end{theorem}

\begin{proof}[Proof of \cref{non-trivial HE}]

We start with an admissible Stein exhaustion $W_0\subset W_1\subset \cdots$ of $W$. Let $\mathfrak s$ represent the asymptotic $\spin^c$ structure on $W$ corresponding to the symplectic 2-form $\omega$. By construction, $\mathfrak s$ restricts to a unique $\spin^c$ structure $\mathfrak s_i$ on $W_i$, which coincides with the Stein cobordism structure on $W_i$. Denoted by $(Y_i,\xi_i)$ the convex boundary of the Stein cobordism $W_i$, we have a cobordism map $\ul{F}^+_{-W_i,\mathfrak s_i;\omega} : \ul{HF}^+_\textnormal{red} (-Y_i, \omega|_{-Y_i}) \to \ul{HF}^+_\textnormal{red} (-Y, \omega|_{-Y})$ such that $F^+_{-W_i,\mathfrak s_i;\omega} (c^+_\textnormal{red}(\xi_i;\omega)) = c^+_\textnormal{red}(\xi;\omega) \neq 0$ \cite{twistedos-sz} where $-W_i$ is consider the upside-down cobordism from $-Y_i$ to $-Y$. 

Now if the interior of $X\cup_Y W$ admits a complete metric with uniformly positive scalar curvature $R\geq 1$, then \cref{exhaustion} implies that there exists an exhaustion $\Omega_1 \subset \Omega_2\subset\dots$ of $X\cup_Y W$ so that each component of $\partial\Omega_i$ is of the form $(\#_j S^3/\Gamma_j) \# (\#_j S^2 \times S^1)$. Moreover, we have $X_i= X\cup_Y W_i$ an exhaustion of $X\cup_Y W$. So by compactness argument, there will exist an integer $k$ and $k'$ such that $X_1 \subset \Omega_{k'} \subset X_{k}$. Thus the map $\ul{F}^+_{-W_k,\mathfrak s_k;\omega}: \ul{HF}^+_\textrm{red} (-Y_k, \omega|_{-Y_k}) \to \ul{HF}^+_\textrm{red} (-Y, \omega|_{-Y})$ factors through $\ul{HF}^+_\textnormal{red} (-\partial\Omega_{k'}, \omega|_{\partial\Omega_{k'}}) = 0$. But $\ul{F}^+_{-W_k,\mathfrak s_k;\omega} (c^+_\textrm{red}(\xi_k;\omega)) = c^+_\textrm{red}(\xi;\omega) \neq 0$, which is a contradiction.
\end{proof}

\subsection{Constructing an admissible Stein Casson Handle}\label{admissible casson handle}

Now, we will elaborate on the process of constructing a Stein-admissible Casson handle, derived from Gompf's method outlined in \cite[Section 3, Proof of Theorem 3.1]{gompf}. Consider $H$, a 4-manifold comprising one 0-handle, one 1-handle, and one 2-handle, denoted as $H=h^0\cup h^1_0\cup h_0^2$. Eliashberg demonstrated \cite{eliashberg} that if the Thurston--Bennequin number of the attaching 2-handle exceeds the smooth framing of its attaching sphere, the manifold admits a Stein structure. However, this inequality might not always hold. Bearing this in mind, we can construct the handle structure of $H$ as follows: Let $X_1= H_1 \cup h_1^1\cup \cdots \cup h_k^1$ for some $k>0$, where $H_1$ represents the 1-skeleton of $H$, to which a few new 1-handles are appended. Then, attach a 2-handle $h_1^2$ to $X_1$, with its attaching sphere passing over $h^1_0$ identical to that of $h_0^2$ but in addition to that this will twist the remaining newly attached 1-handles $h_k^1$, as illustrated in \cref{figure:kinky-handle}, with the smooth framing of $h_1^2$ matching that of $h_0^2$. This process increases the Thurston--Bennequin number of the attaching sphere of $h_1^2$ by $2k$ (bottom picture of \cref{figure:kinky-handle}).

\begin{figure}[htbp]
\centering
\begin{overpic}[scale=.9,  tics=20]{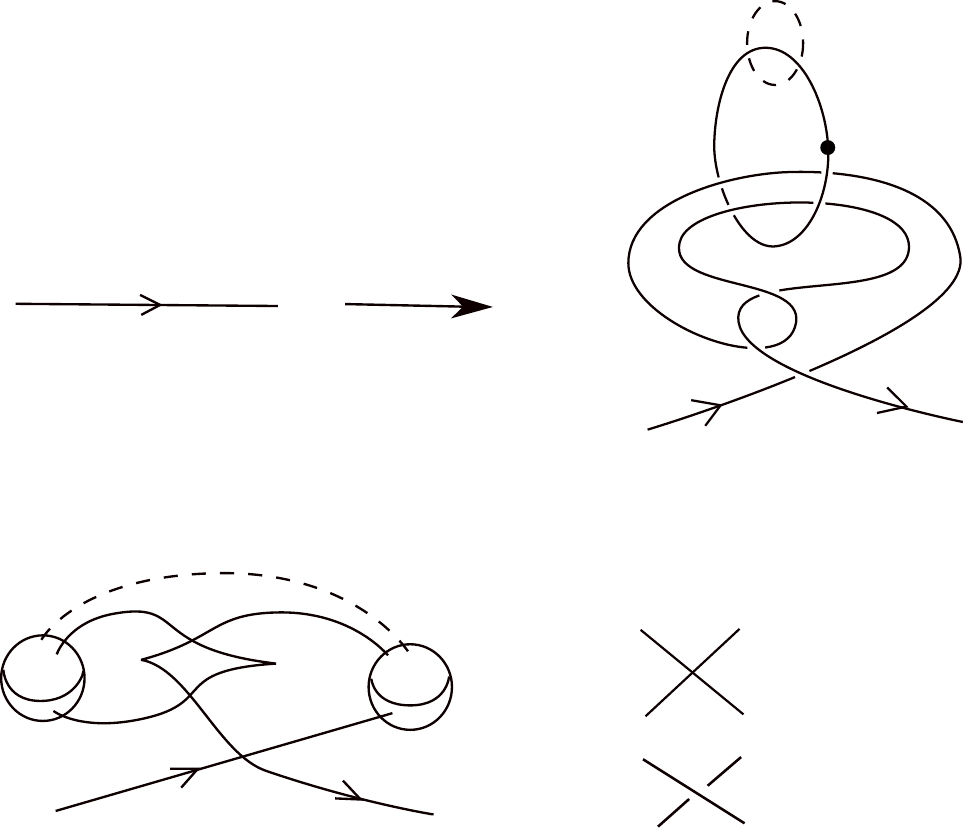}
    \put(50,200) {$h_0^2$}
    \put(370,300) {$h_1^1$}
    \put(420,220) {$h_1^2$}
    \put(140,210) {\tiny{modifying the attaching}} \put(140,195) {\tiny{circle of 2-handle}}
\end{overpic}

\caption{On the top figure we are showing how to modify the attaching 2-handle $h_0^2$ by passing over the new 1-handle $h_1^1$ that increases the Thurston--Bennequin number as shown in the bottom left figure where we can see a Legendrian representation of the new 2-handle $h_1^2$ (where each crossing are as indicated in the right bottom picture).}
\label{figure:kinky-handle}
\end{figure}
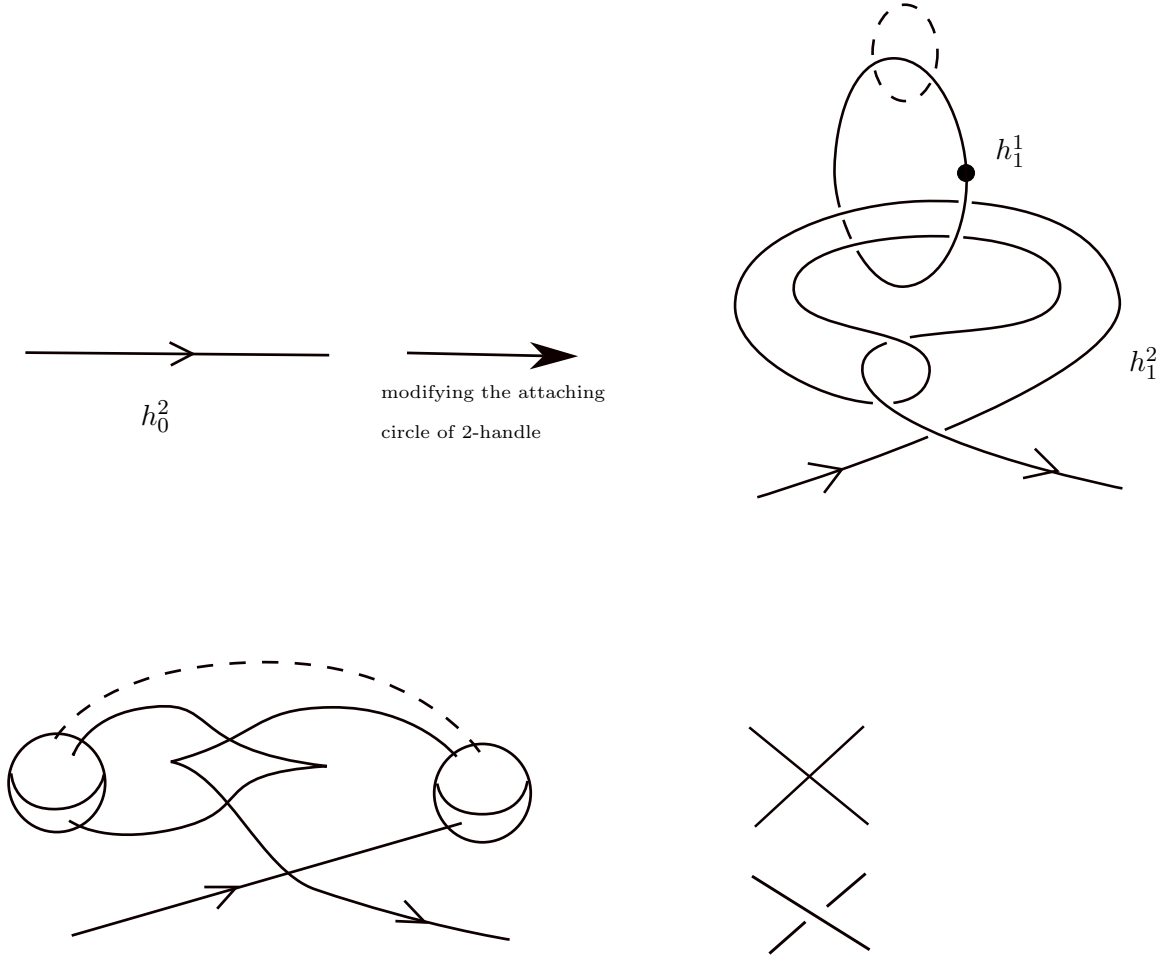

By selecting an appropriate $k$, we ensure $X_1$ becomes Stein. This modification of the original 2-handle $h_0^2$ into the new attaching 2-handle $h_1^2$ is referred to as a kinky-handle. Subsequently, attaching $k$ new 0-framed 2-handles along the meridian of the dotted 1-handles $h_j^1$ for $j>0$ results in a 4-manifold diffeomorphic to $H$. However, the attaching spheres of these 2-handles will have Thurston--Bennequin numbers equal to $0$, aligning with the smooth framing. To establish a Stein structure, we can repeat the above process for each pair of handles by introducing new kinky 2-handles such that the interior admits the Stein structure. This construction can be iterated, adding a third layer of handles onto the second layer, and so forth, generating a manifold $X$ with infinitely many handles, inherently possessing a Stein interior. Note that in this infinite iteration process cancel all the 1-handles with some 2-handles and thus it is simply-connected at infinity and there is no new homology can arise. Additionally, each layer of handle attachments satisfies the conditions of \cref{admissible cobordism}. The work of Casson \cite{casson} and Freedman \cite{Freedman} demonstrated that the aforementioned handle attachment method gives rise to Casson handles, and the resultant 4-manifold $X$ is homeomorphic to the interior of $H$.

\subsection{Proofs} 
We can now construct uncountably many smooth structures on $\R^4$ that don't admit complete metrics with uniformly positive scalar curvature. The basic construction of such exotic $\R^4$ is motivated by the work of Gadgil \cite{gadgil2009open}.
\begin{proof}[Proof of \cref{uncountable R^4}]
Our goal is to construct an exotic $\R^4$ that decomposes as in \cref{non-trivial HE}, namely as $X\cup_Y W$ for $(W,\omega)$ an admissible Stein surface with concave boundary $(Y,\xi)$ so that the reduced contact invariant satisfies $c^+_\textnormal{red}(\xi;\omega)\neq 0$ in $\ul{HF}^+_\textnormal{red}(-Y,\omega)$. 

We begin by constructing $Y$. Let $K$ denote the non-trivial Pretzel slice knot $P(-3,-3,3)$  in $S^3 = \partial B^4$ and $Y= S^3_0(K)$ be the 3-manifold resulting from 0-surgery on $K$. If $\mu$ represents the meridian of $K$ in $S^3$, we observe that $[\mu]$ normally generates $\pi_1(Y)$. Now let us examine the cobordism $X_0$ obtained by attaching a 0-framed 2-handle on $Y\times [0,1]$ along $\mu \subset Y\times \{1\}$. Then $X_0$ is a cobordism from $Y$ to $S^3$. Gabai proved that \cite{gabai} the constructed 3-manifold $Y$ admits a taut foliation. 

We now prove the nontriviality of the reduced contact invariant. Since $Y$ can be equipped with a taut foliation, Eliashberg--Thurston showed that \cite{eliashberg-thurston} after a perturbation of this foliation, it induces a contact structure $\xi$ on $Y$ and moreover there exists a symplectic structure $\omega$ on $Y\times I$ such that both boundary component are convex. Now by capping off one of the boundary components by a symplectic cap as constructed in \cite{mukherjee19} we can construct a symplectic filling of $(Y,\xi)$ with $b^+ >0.$ Thus by \cite[Theorem 4.2]{twistedos-sz} we see that $(Y,\xi)$ is a contact 3-manifold such that reduced contact invariant $c^+_\textnormal{red}(\xi;\omega)\neq 0$ in $\ul{HF}^+_\textnormal{red}(-Y,\omega)$.

Now we will construct our desired compact 4-manifold $X$ with boundary $Y$ as follows; let $D$ be a slice disk bounded by $K$ in $B^4$. Since $H_2(B^4\Z)=0$, we know that the algebraic self-intersection number of $D$ is zero. Thus the boundary of the $4-$manifold $X = B^4 \sm \nu(D)$ is diffeomorphic to 0-framed Dehn surgery along the knot $K$. So we have our desired compact 4-manifold $X$ with boundary $Y$. See \cref{figure:ribbon} for a Kirby-diagram of $X$. 

\begin{figure}[ht]
\centering

\begin{overpic}[scale=.9,  tics=20]{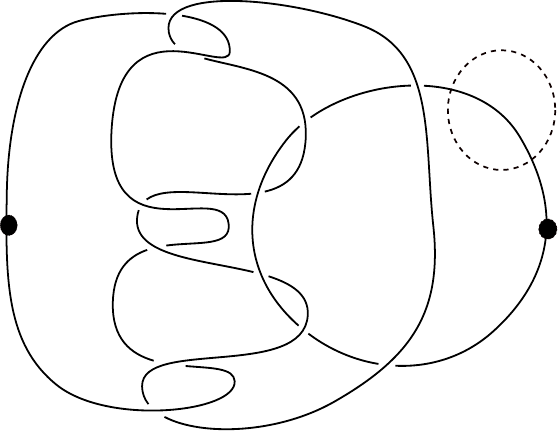}
    \put(160,180){0}
\end{overpic}

\caption{The above descrubed a Kirby picture of the 4-manifold obtained by deleting a neighborhood of a slice disk bounded by the Pretzel knot $P(-3,-3,3)$. And the dashed-circle denoted the meridian of the Pretzel knot.}
\label{figure:ribbon}
\end{figure}

Now let us focus on the open 4-manifold $R = B'\cup W_0$ where $W_0 = X_0\sm S^3$. Gompf showed in \cite[Theorem 3.4]{gompf} that by replacing the open 2-handle of $W_0$ with Stein Casson handles (attaching them along the dashed circle in \cref{figure:ribbon}), one can construct uncountable many non-diffeomorphic 4-manifolds $R_c = B'\cup CH_c$, indexed by Cantor set (compare with \cite{df,BZ}). These $R_c$'s are seen to have trivial homology in all dimensions. Furthermore, $\pi_1(R_c)= 1$, and $R_c$ is simply-connected at infinity. According to Freedman \cite{Freedman}, $R_c$ is thus homeomorphic to $\mathbb{R}^4$. Since $CH_c$ is a Stein Casson handle, by the result of Eliashberg, it has a handle decomposition $\{h^i_c\}$ with 1- and 2-handles. 
Moreover, we can arrange the handles $h_c^i$ (as explained in \cref{admissible casson handle}) to satisfy \cref{admissible cobordism}. Therefore we have an admissible cobordism $\{X_c^i\}$ of $R_c$, where $X_c^0= X$, and $X_c^i$ obtained from $X_c^{i-1}$ by attaching a Stein 1- or 2-handle $h_c^i$ along $\partial X_c^{i-1}$. Let $\omega$ be the symplectic 2-form of $CH_c$ obtained by the Stein structure. Hence if we consider $W= CH_c$ then by \cref{non-trivial HE}, $R_c$ doesn't admit a complete positive scalar curvature metric $R\geq 1$.
\end{proof}

Similarly, we can show that if $W$ is any smooth $4$-manifold with boundary then the interior of $W$ admits a smooth structure with no complete Riemannian metric with uniformly positive scalar curvature. 
\begin{proof}[Proof of Theorem \ref{thm:compact}]
In \cite{gadgil2009open}, Gadgil constructed an exotic $\mathbb{R}^4$ with non-trivial end Floer homology. Although Gadgil used the $HF^+$ version of the Heegaard Floer homology, but the same argument of the proof is valid on $HF^-$ since there is a natural and canonical map between $HF^+_\textnormal{red}\to HF^-_\textnormal{red}$, specifically, he showed that $R$ admits an admissible exhaustion $R_0 \subset R_1 \subset R_2 \cdots$ and an asymptotic $spin^c$ structure $\mathfrak{s}$ such that $\operatorname{Lim}_{i\to \infty}\ul{F}^-_{R_i,\mathfrak{s_i}}(x_i) \neq 0$ where $x_{k+1}:= \ul{F}^-_{R_k,\mathfrak{s}_k}(x_k) $ defined inductively for some fix $x_1\in \ul{HF}^-_\textnormal{red}(\partial R_1; \mathfrak{s}_1)$. To construct our candidate for a smooth structure on the interior of $W$, we will perform an end sum operation between the interior of $W$ and $R$ as follows. Consider the interior of $W$ as $W \cup Y \times [1, \infty)$ and fix an arc $\gamma_1 = {y} \times [0, \infty) \subset Y \times [1, \infty) \subset \mathring{W}$. On the other hand, pick a simple arc $\gamma_2: [1, \infty) \to R$ such that the sub-arc $\gamma_2([i, i+1]) \subset R_i$ connects the two boundary components of $R_i$. Now, construct a new open 4-manifold $W' = \mathring{W} \setminus \mathrm{nbd}(\gamma_1) \cup R \setminus \mathrm{nbd}(\gamma_2)$ (this operation is known as the end sum operation). Note that $W'$ has the same homeomorphism type as $\mathring{W}$. However, by construction, $W'$ admits an admissible cobordism $W'_0 \subset W'_1 \subset W'_2 \cdots$ where $W'_i = Y \times [i, i+1] \setminus \mathrm{nbd}(\gamma_1([i, i+1])) \cup R_i \setminus \mathrm{nbd}(\gamma_2([i, i+1]))$. We can think of $W_i'$ as a cobordism from $Y\# \partial R_i$ to $Y\# \partial R_{i+1}$. Now by Kunneth formula for connected sum \cite[Theorem 1.5]{OzsvathSzabo03}, we know that $\ul{HF}^-(Y_1\# Y_2, \mathfrak{s}_1\#\mathfrak{s}_2)\otimes_{\Z} \Q = \ul{HF}^-(Y_1,\mathfrak{s}_1)\otimes_{\Z[U]} \ul{HF}^-(Y_2,\mathfrak{s_2})\otimes_{\Z} \Q$. Let us fix a $spin^c$ structure $\mathfrak{s_0}$ on $Y\times [1,\infty)$ and since the cobordism $Y\times [i,i+1]$ is trivial, the induced cobordism map by $W_i'$, $\ul{F}^-_{W_i',\mathfrak{s_0}\# \mathfrak{s_i}} : \ul{HF}^-(Y,\mathfrak{s_0})\otimes_{\Z[U]} \ul{HF}^-(\partial R_i,\mathfrak{s_i})\otimes_{\Z} \Q \to \ul{HF}^-(Y,\mathfrak{s_0})\otimes_{\Z[U]} \ul{HF}^-(\partial R_{i+1},\mathfrak{s_i})\otimes_{\Z} \Q$ acts as $\ul{F}^-_{W_i',\mathfrak{s}\# \mathfrak{s_i}}(y\otimes x_i\otimes q)= (y\otimes \ul{F}^-_{R_i, \mathfrak{s_i}}(x_i)\otimes q).$ Now note that if $y$ is a generator of the tower of $\ul{HF}^-(Y,\mathfrak{s_0})$ and $x\in \ul{HF}^-_\textnormal{red}(Y_i,\mathfrak{s_i})$, then $(y,x,q)$ is an element of $\ul{HF}^-_\textnormal{red}(Y\# Y_i, \mathfrak{s}_0\#\mathfrak{s}_i)$. And thus the end Floer homology, $\operatorname{Lim}_{i\to \infty} \ul{F}^-_{W_i',\mathfrak{s_0}\# \mathfrak{s_i}}(y\otimes x_i\otimes q)\neq 0$. So by \cref{HF-} and \cref{thm:end-floer-obstruction}, we can conclude that $W'$ does not admit uniformly positive scalar curvature metric. 
\end{proof}

\appendix
\section{Overview of scalar curvature}\label{app:survey}

This appendix contains a brief survey of the topological study of scalar curvature. Many other more comprehensive surveys exist, including \cite{Rosen06,Gromov23,Chodosh21}.

 The scalar curvature of a Riemannian manifold $(M^n,g)$ is a function $R:M \rightarrow \mathbb{R}$ such that for every point $p\in M$, the volume of balls has the following infinitesimal expansion:
\begin{equation}\label{eq:defi-scalar-curv}
\vol_M (B_\varepsilon(p))= \vol_{\mathbb {R}^n}(B_\varepsilon(0))\left(1-\frac{R(p)}{6(n+2)}\varepsilon^2+O(\varepsilon^2)\right).
\end{equation}
This is the weakest curvature invariant one can define on a manifold and it is well-known to be very flexible in some ways but rigid in others. On any smooth manifold, any function that is negative at some point can be prescribed as the scalar curvature of some metric (c.f. \cite[Theorem 0.1]{Rosen06}) but, remarkably, there are global obstructions a manifold must satisfy to carry a metric with positive scalar curvature.

\subsection{Obstructions} All known obstructions arise from one of three methods. The first, discovered by Lichnerowicz \cite{Lich63}, employs the Atiyah--Singer index theorem for the Dirac operator. He demonstrated that any closed spin manifold with a non-vanishing $\hat{A}$-genus cannot support a metric of positive scalar curvature. This, for example, implies that the $K3$ surface cannot accommodate such a metric. A refinement by Hitchin \cite{Hit74} further rules out the existence of positive scalar curvature metrics on certain exotic spheres of dimensions $8k+1,8k+2$ for all $k\geq 1$. For a generalization see \cite{GrLa83}.

 The next obstruction, due to Schoen--Yau \cite{ScYa79a}, uses the second variation of the area of minimal hypersurfaces. They discovered it while investigating the positive mass conjecture in general relativity \cite{PMT}. Their method works for manifolds of dimension $3\leq n\leq 7$. It obstructs positive scalar curvature on manifolds satisfying a certain condition on their integral cohomology ring, a class that includes closed manifolds admitting non-zero degree maps to a torus.

Finally, in dimension 4, Witten \cite{Wit94} demonstrated that closed manifolds $M$ with positive scalar curvature and a second Betti number $b^{+}_2(M) > 1$ must have vanishing Seiberg-Witten invariants. However, as shown by Taubes \cite{Taubes94}, this is not possible if such manifolds admit a symplectic structure, thereby obstructing the existence of positive scalar curvature metrics on a significant number of closed smooth 4-manifolds. Furthermore, for 4-manifolds with periodic ends, various generalizations of Seiberg-Witten theory also obstruct the existence of positive scalar curvature metrics \cite{KM,KTJtop,Lin,LRN}. 

\subsubsection{$\mu$-bubbles} Recently, Gromov introduced  \cite{Gromov23} a new idea that generalizes the minimal hypersurface technique to other situations (cf.\ Section \ref{sec:mu-bubble}) by localizing the minimizer via distance function. See also \cite{Zhu:rigidity,Zhu:width,zhu2021,Zhu:calabi,Rade23}. 

Using the $\mu$-bubble technique, Chodosh--Li proved that closed aspherical $n$-manifold
does not admit metric with positive scalar curvature for $n=4,5$. The $n=5$ was independently obtained by Gromov \cite{Gromov20}. Other obstructions using incompressible surfaces were obtained by \cite{zhu2021, CRZ23}. We also note that similar localizations have been obtained for spinors by various authors \cite{Z:band,Z:w,GXY,Cecchini:longneck,CZ:mububspin}. 

\subsection{(Partial) classification results} 
In $n=3$ dimensions, the minimal hypersurface technique can be used to impose strong restrictions on the fundamental group of closed manifolds with positive scalar curvature \cite{ScYa79b}. After Perelman solved the Poincar\'{e} conjecture \cite{perelman}, a complete classification of closed 3-manifolds with positive scalar curvature up to diffeomorphism could be obtained (cf. Proposition \ref{prop:dominate3d}).

In dimensions $n\geq 5$, the problem is also understood for simply connected manifolds: Gromov-Lawson \cite{GL80} showed that every closed simply connected non-spin manifold of dimension $n\geq5$ admits a positive scalar curvature metric, and Stolz \cite{Stolz} showed that spin manifolds can only admit such metrics if their $\alpha$-index invariant vanishes.

Also using $\mu$-bubbles, Chodosh--Li--Liokumovich \cite{CCL23} showed that if $M$ is a closed manifold of dimension $n = 4$
(resp. $n = 5$) with $\pi_2(M) = 0$ (resp. $\pi_2(M) = \pi_3(M) = 0$) that admits a
metric of positive scalar curvature, then a finite cover $\widehat M$ of $M$ is homotopy equivalent to $S^n$ or connected sums of $S^{n-1}\times S^1$. In a different vein, Bamler--Li--Mantoulidis showed that if $M$ is a closed smooth $4$-manifold that admits a metric of positive scalar curvature then it can be obtained by performing $0$- and $1$-surgeries on a disjoint union of PSC orbifolds with first Betti number $b_1=0$ \cite{BamlerLiMant}.

  \section{Perturbed Heegaard Floer homology}\label{perturbedHF}
Recall that the Novikov ring over $\F_2$ is a set of formal series 
    \[
        \Lambda = \left\{\sum_{x\in\R}n_x z^x : n_x\in\F_2\right\} 
    \] 
    where the set 
    \[
        \{x \in (-\infty, c] : n_x\neq0\}
    \] 
    is finite for every $c\in\R$. One may easily check that this is a field under the obvious operations. 

Let $Y$ be a closed $3$--manifold equipped with a closed $2$-form $\omega\in\Omega^2(Y)$. Then there exists an action of a group ring $\F_2[H^1(Y)] \cong \F_2[H_2(Y)]$ on $\Lambda$, induced by $\omega$, which is defined as follows. For $a \in H_2(Y)$,

\[
    e^a \cdot z^x = z^{x+\int_a\omega}.
\] 

The naturality of perturbed Heegaard Floer homology is conveniently described by projective transitive systems, which were first introduced by Baldwin and Sivek in \cite{BS15}.

 Recall that Heegaard Floer homology of a 3-manifold $Y$ with a $\spin^c$ structure is a of $\F_2[U]$-modules $HF^\circ(Y,\frs)$ for $\circ\in\{\infty,+,-,\wedge\}$, which fit into the following two long exact sequence
\[
    \cdots\xrightarrow{\tau}HF^-(Y,\frs)\xrightarrow{\iota} HF^\infty(Y,\frs)\xrightarrow{\pi} HF^+(Y,\frs)\xrightarrow{\tau} HF^-(Y,\frs)\rightarrow\cdots
\]

Recall that associated to this long exact sequence there is another 3-manifold invariant 
 \[
 HF^+_\textnormal{red}(Y,\mathfrak{s})= Coker(\pi) \cong Ker(\iota) = HF^-_\textnormal{red}(Y,\mathfrak{s})
 \]
 The isomorphism in the middle is induced by the co-boundary map. Recall that $d(Y,\mathfrak{s})$ is the minimum grading of the torsion-free elements in the image of $\{\pi : HF^{\infty}(Y,\mathfrak{s})\rightarrow HF^+(Y,\mathfrak{s})\}$.

In \cite{OS04a}, Ozsv\'{a}th--Szab\'{o} introduced Heegaard Floer homology perturbed by a second real cohomology class, which is more thoroughly discussed in \cite{JM08}. For completeness we will discuss the construction now- Let $\omega\in\Omega^2(Y)$ be a closed 2-from on $Y$ and $\cH = (\Sigma,\alpha,\beta,w)$ an $\frs$-admissible pointed Heegaard diagram of $Y$. Denote the two handlebodies determined by $\cH$ by $H_{\alpha}$ and $H_{\beta}$ respectively and let $D_{\alpha}$ and $D_{\beta}$ be sets of compressing disks of $H_{\alpha}$ and $H_{\beta}$, respectively, such that $D_{\alpha}$ intersects $\Sigma$ along $\alpha$ and $D_{\beta}$ intersects $\Sigma$ along $\beta$. Note that $\phi\in \pi_2(x,y)$ determines a 2--chain $\cD(\phi)$ on $\Sigma$ with boundary a union the loops in $\alpha\cup \beta$. One may cone these loops in the compressing disks $D_{\alpha}$ and $D_{\beta}$ to obtain a 2--chain $\widetilde{D}(\phi)$. Now we define
\[
    A_\omega(\phi):=\int_{\widetilde{D}(\phi)}\omega.
\]

Consider a chain complex $CF^\infty(\cH,\frs;\omega)$ which is a free $\Lambda$-module generated by $U^ix$ for $x\in T_\alpha\cap T_\beta$, such that the $\spin^c$ structure associated to $x$ is $\frs$, i.e. $\frs(x)=\frs$. The differential is defined as follows.

\[
    \partial^\infty(U^ix) = \sum_{y\in T_\alpha\cap T_\beta}\sum_{\substack{\phi\in\pi_2(x,y) \\ \mu(\phi)=1}}\#(\cM(\phi)/\R)\cdot z^{A_\omega(\phi)}\cdot U^{i-n_w(\phi)}y \quad \text{(mod 2)},
\]
where $n_w(\phi)$ is the algebraic intersection number between $\phi\in\pi_2(x,y)$ and $\{w\}\times \Sym^{g-1}(\Sigma)$ and $\cM(\phi)$ is the space of $J$-holomorphic disks in the homotopy class $\phi$ for $J \in \mathcal{J}$, where $\mathcal{J}$ is a generic family of almost complex structures on $\Sym^g(\Sigma)$. For $\circ\in\{+,-,\wedge\}$, $\partial^\circ$ is induced from $\partial^\infty$ in the usual way and $(\partial^\circ)^2=0$ as in the original Heegaard Floer homology. Now we define 
\[
    HF^\circ(Y,\frs;{\omega}):=H_*(CF^\circ(\cH,\frs;\omega),\partial^\circ).
\]
The homology $HF^\circ(Y,\frs;{\omega})$ depends on the choice of $\cH$ and $J$, but Juh\'{a}sz and Zemke in \cite{JZ18} proved the well defineness.

We have the following functoriality in Heegaard Floer homology groups induced by cobordisms. 

\begin{theorem}[Ozsv\'ath--Szab\'o~\protect{\cite[Section~3.1]{twistedos-sz}}]
    Let $W$ be a cobordism from $Y_1$ to $Y_2$. Suppose $\omega$ is a closed $2$-form on $W$ and $\frs\in \spin^c(W)$ is a $\spin^c$ structure on $W$. Then the cobordism map
    \[
        F^\circ_{W,\frs;\omega}\colon HF^\circ(Y_1,\frs|_{Y_1};{\omega|_{Y_1}})\to HF^\circ(Y_2,\frs|_{Y_2};{\omega|_{Y_2}})
    \]
    is well-defined up to overall multiplication by $z^x$ for $x\in\R$.
\end{theorem}

Consider $\omega\in\Omega^2(W,\partial{W})$ which is a closed $2$-form on $W$ compactly supported in the interior of $W$. We will say such an $\omega$ is a $2$-form on $(W,\partial{W})$. Let $W$ be a cobordism from $Y_0$ to $Y_1$ equipped with a closed $2$-form $\omega$ and $\frS\subseteq \spin^c(W)$ a subset of $\spin^c$ structures on $W$ such that each $\frs\in\frS$ has the same restriction to $\partial{W}$. 

If $\circ\in\{\infty, -\}$, we further assume that there exists only finitely many $\frs\in\frS$ such that $F^\circ_{W,\frs;\omega}$ is non-vanishing. Then, there exists a cobordism map
\[
    F^\circ_{W,\frS;\omega}: HF^\circ(Y_1;{\omega|_{Y_1}})\rightarrow HF^\circ(Y_2;{\omega|_{Y_2}}),
\]
which is also well-defined up to overall multiplication by $z^x$ for $x\in\R$. Although addition is not well-defined in projective systems, we may find representatives of $F^\circ_{W,\frS;\omega}$ for $\frs\in\frS$ so that
\[
F^\circ_{W,\frS;\omega} \doteq \sum_{\frs\in \frS} F^\circ_{W,\frs;\omega}.
\]

If $\omega=\{\omega_1,\ldots,\omega_n\}$ is an $n$-tuple of closed $2$-forms on a $3$--manifold $Y$, we can define a $\Lambda_n[U]$-module $HF^\circ(Y,\frs;{\omega})$ as above, where $\Lambda_n$ is the $n$-variable Novikov ring over $\F_2$. All the theorems and lemmas in this section hold for this version.

Let $a = (a_1,\ldots,a_n)$ be an $n$-tuple of integers. We will use the notation
\[
    z^{a} := z_1^{a_1} \cdots z_n^{a_n}.
\]

\begin{lemma}[Juh\'{a}sz--Zemke \protect{\cite[Lemma~3.4]{JZ18}}]\label{lem:triviallytwisted}
    Let $W$ be a cobordism from $Y_1$ to $Y_2$ and $\omega=\{\omega_1,\ldots,\omega_n\}$ be an $n$-tuple of closed $2$-forms on $(W,\partial{W})$. Suppose $\frS \subseteq \spin^c(W)$ is a subset of $\spin^c$ structures on $W$. If $\circ \in \{-,\infty\}$, we further assume that there are only finitely many $\frs \in \frS$ where $F^\circ_{W,\frs} \neq 0$. Fix an arbitrary $\spin^c$ structure $\frs_0 \in \spin^c(W)$. Then,
    \[
        F^\circ_{W,\frS;\omega}\doteq\sum_{s\in\frS} z^{\langle i_*(\frs-\frs_0)\cup[\omega],[W,\partial W]\rangle} \cdot F^\circ_{W,\frs},
    \]
where $i_*:H^2(W;\Z)\to H^2(W;\R)$ is induced by the inclusion $i:\Z\to\R$. 
\end{lemma}

\bibliographystyle{amsalpha}
\bibliography{references}
\end{document}